\documentclass{amsart}
\usepackage[utf8]{inputenc}

\usepackage[english]{babel}
\usepackage{geometry}
\usepackage{xcolor}
\usepackage{amssymb}
\usepackage[all]{xy}

\makeindex

\usepackage{color}

\newtheorem{thm}{Theorem}[section]
\newtheorem{thm*}[thm]{Theorem*}

\newtheorem{cor}[thm]{Corollary}
\newtheorem{cor*}[thm]{Corollary*}

\newtheorem{prop}[thm]{Proposition}
\newtheorem{prop*}[thm]{Proposition*}

\theoremstyle{definition}
\newtheorem{defn}[thm]{Definition}
\newtheorem{example}[thm]{Example}

\newtheorem{rem}[thm]{Remark}   
\newtheorem{rem*}[thm]{Remark*}           

\newtheorem*{ack}{Acknowledgments}      

\newtheorem{defn-thm}[thm]{Definition--Theorem}  
\newtheorem{defn-lem}[thm]{Definition--Lemma}  

\theoremstyle{remark}

\renewcommand{\c}[0]{{\mathbb C}}  

\renewcommand{\o}[0]{{\mathcal O}} 

\renewcommand{\r}[0]{{\mathbb R}}

\newcommand{\p}[0]{{\mathbb P}}

\newcommand{\F}[0]{{\mathbb F}}

\newcommand{\im}[0]{\operatorname{im}}

\newcommand{\Seg}{\operatorname{Seg}}

\newcommand{\C}{\mathbb{C}}

\def\loccoh#1.#2.#3.#4.{H^{#1}_{#2}(#3,#4)}

\DeclareMathAlphabet{\mathchanc}{OT1}{pzc}%
                                {m}{it}

\newcommand{\sym}[0]{\operatorname{Sym}}

\numberwithin{equation}{section}

\begin{document}







\title[Catalecticant intersections and confinement]{Catalecticant intersections and confinement of decompositions of forms}

\author[E. Angelini, C.Bocci, L. Chiantini]{Elena Angelini, Cristiano Bocci, Luca Chiantini}
\address[E.Angelini,  C. Bocci, L. Chiantini]{Dipartimento di Ingegneria dell'Informazione e Scienze Matematiche
\\ Universit\`a di Siena\\ Via Roma 56\\ 53100 Siena, Italia}
\email{elena.angelini@unisi.it, cristiano.bocci@unisi.it, luca.chiantini@unisi.it}

\begin{abstract} 
We introduce the notion of confinement of decompositions for forms or vector of forms. The confinement, when it holds, 
lowers the number of parameters that one needs to consider, in order to find all the possible decompositions of a given set of data. 
With the technique of confinement, we obtain here two results.
First, we give a new, shorter proof of a result by London (\cite{London90}) that $3$ general plane cubics have $2$ simultaneous Waring decompositions of rank $6$. 
Then we compute, with the software Bertini, that $4$ general plane quartics have $18$ different decompositions of rank $10$ (a result which was not known before).\\

\end{abstract}

\maketitle

\section{Introduction}\label{sec:intr}

A classical geometric approach for the study of the number of  
minimal decompositions  of general tensors, in terms of rank $1$ elements, goes back to Terracini and has been
pointed out in \cite{CCi06} and \cite{COtt12}. 
 The method is based on the observation that when a tensor of \emph{sub-generic
rank} has several minimal decompositions, then all the corresponding finite sets are bounded to specific
sub-varieties, the {\it contact loci} of the variety $X$ of rank $1$ tensors. 
Since contact loci determine families of degenerate varieties that cover $X$, their existence has many 
geometric consequences, so that one can detect the existence of contact loci by analyzing the structure of $X$.

\indent Families of sub-varieties which determine a special behavior are rather natural in interpolation theory. 
For interpolation problems on the plane, Segre's conjecture and its many modern variants predict 
that the superabundance of linear systems with prescribed singularities is equivalent to the existence of special varieties passing through 
sets of points  (see e.g. \cite{Ci01}).
Segre's idea that sub-varieties can determine special behavior for interpolation problems has been
generalized in higher dimensional spaces by the second author, who introduced the notion of {\it special effect varieties}
(see \cite{Bocci05}). These varieties are responsible for the existence of interpolating linear systems whose
dimension is higher than the expected value.
The strict connection 
between superabundant interpolating systems for finite sets of points of multiplicity  $2$ and unexpected number of minimal decompositions
is described in details in Section 2 of \cite{CCi06}. It turns out that an unexpected number of minimal decompositions implies the existence
of positive dimensional contact loci, which are degenerate subvarieties of Segre or Veronese varieties, and conversely. Theorem 2.9 of \cite{CCi06} shows that
contact loci provide the first examples of special effect varieties, for multiplicity $2$ interpolation problems. 

Terracini's idea cannot be extended to tensors of generic rank. Indeed, contact loci are defined by
hyperplanes which are tangent at $r$ points to the variety $X$ of rank $1$ objects, where $r$ is the value of the rank
under analysis. By definition, if $r$ is the {\it generic rank}, then there are no hyperplanes tangent to $X$ at $r$ points,
thus the notion of contact locus is meaningless for such value of $r$.

On the other hand, in the study of tensors $T$ of generic rank, as soon as the number of minimal decompositions
is small, then special sub-varieties that enclose all possible decompositions of $T$ arise quite naturally in
many examples (built both theoretically or with the aid of computer algebra packages).
In this respect, an evidence that  families of degenerate sub-varieties play a fundamental
role is explained in the paper of Ciliberto and Russo \cite{CiRusso11}, where 
degenerate rational curves are shown to determine the uniqueness of secant varieties through
generic points, for some special classes of varieties $X$ (see also \cite{PirioRusso13}, Corollary 5.4 and Remark 5.5).\\

One can  guess  that special sub-varieties are responsible for all special behaviors of the decompositions of a general tensor of rank $r$
(note that for sub-generic ranks a special behavior means to have more than one minimal decomposition, while for the generic rank it
means that the number of decompositions is small).

\indent In the case of symmetric tensors $f$ defined over a field $\F$, which will be identified with forms (homogeneous polynomials), 
there exists another approach that can produce sub-varieties bounding the decompositions of $f$.
Indeed, for a degree $d$ form $f$ in $n+1$ variables over $\F$, the contraction by $f$ determines a map $C_f$ 
(the {\it catalecticant map} of \cite{IK})
from the space of polynomial derivatives  of order $h$, which can be identified with $(\sym^{h}\F^{n+1})^{\vee}$,
to the space of forms of degree $d-h$. If $f= \lambda_{1}L_1^d+\dots +\lambda_{k}L_k^d$ is a Waring decomposition 
(the $L_i$'s and $ \lambda_{i} $'s being, respectively, linear forms and scalars), then it is immediate to see that $\p(\im(C_f))$ 
sits into the span $\Lambda$ of the powers $[L_1^{d-h}],\dots, [L_k^{d-h}]$ in the space of forms of degree $d-h$.
If the matrix of $C_f$ has rank $k$ and the $[L_i]$'s are linearly independent, then $\p(\im(C_f))=\Lambda$.
Thus, for any other Waring decomposition $f=\mu_{1}M_1^d+\dots +\mu_{k}M_k^d$, the images of the $M_i$'s under
the $d-h$ Veronese map  are contained in the intersection of $\Lambda$ with
the corresponding Veronese variety $ X^n_{d-h} $.

The intersection $\Lambda\cap X^n_{d-h}$ has been studied until now only in order to prove that a form $f$
is {\it identifiable}, i.e. its minimal Waring decomposition is unique. Namely, in the previous setting, 
if $\Lambda\cap X^n_{d-h}$ consists of $k$ points, then the uniqueness (modulo scaling and
reordering) follows immediately, for a specific $f$. See e.g. \cite{MassaMellaStagliano18} for an account of the method.
 
We propose to extend the method to the case where the intersection
$Y=\Lambda\cap X^n_{d-h}$ is not finite, but yet
$Y$ is a  special subvariety of the space of tensors of rank $1$. Since, in our setting, all the minimal
 decompositions of $f$ must lie on $Y$ and $f$ itself is a point in the projective span of $Y$,
 then if we can compute the number of secant spaces to $Y$ passing through a general point of the span,
 then we get the number of different Waring decompositions of a general $f$.

We present two cases in which the method yields  interesting conclusions. 
 
First (see Theorem \ref{thm:collect1})  we analyze the simultaneous decompositions of $ 3 $ 
forms of degree $3$ in $3$ variables. In this case, following \cite{AngeGaluppiMellaOtt18} Section 2.2, 
we know that the variety of elements of rank $1$ corresponds to the Segre-Veronese embedding of $\p^2\times\p^2$ 
of bi-degree $(1,3)$. This case has been classical considered by London in \cite{London90}. London used a direct method
 for the analysis of the net of quadrics, to determine that a general triple of cubic forms in $3$ variables
 has two minimal simultaneous decompositions. London's approach is so involved that the result has been fully 
 reconsidered and proved by Scorza in \cite{Scorza99}. We present a new, elegant proof of London's result, which makes use
 of our analysis.
 
We observe that the analysis easily extends to simultaneous decompositions of forms of degrees $3,3,3,2, \ldots, 2$ in $3$ variables (see Corollary \ref{thm:collect2}). 
Also in this case, we get that the number of such minimal decompositions is still $ 2 $. The result stated in Corollary \ref{thm:collect2} holds, in particular, 
for the case of quadruples of degrees $3,3,3,2$ in $3$ variables (see Proposition \ref{thm:collect2}); this case has been analyzed only recently by 
Galuppi, Mella, Ottaviani and the first author in \cite{AngeGaluppiMellaOtt18} Section 4, where, after a computer aided analysis by means of the software Bertini (\cite{Bertini}), 
it is conjectured that the number of simultaneous decompositions is $2$.   We can prove that the guess is correct.

 In all these situations, the intersection variety  $Y$ that we find via the catalecticant map
 (which has the role of a {\it special effect variety} for our Waring problem) is an elliptic normal curve. Thus
 we conclude that the number of minimal decompositions is $2$, since  elliptic normal curves are known to have 
 two secant spaces through a general point of the span (see Proposition 5.2 of \cite{CCi06}).

Let us notice that, even for specific tensors or polynomial vectors $f$,  if the catalecticant map determines an intersection variety $Y\subsetneq X$,
then the study of the Waring decompositions of $f$ is reduced to the study of secant spaces
to $Y$ passing through $f$. When the variety $Y$ is well known, we get in this way a lot of information.
For instance, when $Y$ is an elliptic curve (see also Example \ref{ex:sextics}), by means of the main result Theorem 4.2 in \cite{AngeBocciC18}
we can deduce the uniqueness of simultaneous decompositions of $3$ cubics in three variables,  {\it over the real field $\r$},
in a non-empty euclidean open subset.
 
Next (see Theorem \ref{thm:4444})  we analyze the simultaneous decomposition of $ 4 $ general
forms of degree $4$ in $3$ variables. In this case, following \cite{AngeGaluppiMellaOtt18} Section 2.2, 
we know that  the simultaneous rank is $10$, and the variety of elements of rank $1$ corresponds to the Segre-Veronese embedding of $\p^3\times\p^2$ 
of bi-degree $(1,4)$. The number of simultaneous decompositions of a general vector of $4$ quartics was previously unknown: a
direct computational approach did not succeed in providing the number. We use the confinement to a general
3-fold section of the Segre embedding of $\p^3\times\p^2$ (whose birational description is the target of the (short) Section 5) to decrease the
number of parameters. In this new setting, the software Bertini can compute in seconds that the
number of different simultaneous Waring decomposition is $18$ (all the computations have been done by means of Desktop PC - Windows 8 
operating system x64 - Intel Core i5 CPU 1.70 GHz - 4 GB RAM). 

As usual, similar computations by Bertini do not provide a theoretical
proof that the number of decompositions is $18$, but they give us a strong evidence of the result. See \cite{HauensteinOedOttSomm19} for a discussion
on the reliability of Bertini's computations. In Section 6 we will distinguish between theoretically defined results and results for which direct computations
provide an extremely strong evidence by putting in front of them an asterisk * (see e.g. Theorem* \ref{thm:4444}). 

As above, we observe that the analysis easily extends to simultaneous decompositions of forms of degrees 
$4,4,4,4,3, \ldots, 3$ in $3$ variables (see Corollary \ref{thm:4444ext}). 
It is immediate to get that the number of such minimal decompositions is $18$ in this case too.

Let us observe that when $m>4$, for vectors of $m$ forms of degree $m$ in $3$ variables the catalecticant confinement of decompositions
does not hold, so we cannot decrease the number of parameters as before. We believe, however, that many other interesting initial cases
of decompositions of forms or vector of forms can be handled with our technique. 

The paper is structured as follows. In section \ref{sec:preliminaries} we introduce the simultaneous Waring decompositions, 
recalling the main notations and definitions (for a detailed description we refer to \cite{AngeGaluppiMellaOtt18}). 
In section \ref{sec:ellcurv} we show how elliptic normal curves interact with the varieties of rank $ 1 $ elements introduced 
in Example \ref{rem:London} and Example \ref{rem:generalcase}. In section \ref{sec:cat} we explain how certain catalecticant maps combine 
with elliptic normal curves to prove our Theorem \ref{thm:collect1}. Section \ref{sec:p3p2} is devoted to the description of linear sections
of the Segre embedding of $\p^3\times\p^2$. Section \ref{sec:cat4} is devoted to the determine the evidence that
$4$ plane quartics have $18$ different decompositions of rank $10$.

\section{Simultaneous Waring decompositions}\label{sec:preliminaries}

Let $ \mathbb{F} $ be either the complex or real field. Let $ n,r \in \mathbb{N} $ and let $ a_{1}, \ldots, a_{r} \in \mathbb{N} $.
 Inside of $ \mathbb{F}[x_{0}, \ldots, x_{n}] $, which denotes the space of forms in $ n+1 $ variables 
with coefficients in $ \mathbb{F} $, we consider the elements of degree $ a_{j} $, that is the sub-space 
$ \mathbb{F}[x_{0}, \ldots, x_{n}]_{a_{j}} \cong \sym^{a_{j}} \mathbb{F}^{n+1} $, for $ j \in \{1, \ldots, r\} $.    

\begin{defn} (See Section 4 of \cite{AngeGaluppiMellaOtt18})
A \emph{polynomial vector} (or \emph{multiform}) is a vector $ f = (f_{1}, \ldots, f_{r}) $ such that, for all $ j \in \{1, \ldots, r\} $, 
$ f_{j} \in \mathbb{F}[x_{0}, \ldots, x_{n}]_{a_{j}} $. 

A polynomial vector
$ f $ is said to be \emph{general} if each component  $ f_{j} $ is a general element of $ \mathbb{F}[x_{0}, \ldots, x_{n}]_{a_{j}} $. 
\end{defn}

\begin{defn}
A \emph{simultaneous Waring decomposition} of $ f $ over $ \mathbb{F} $ is given by linear forms 
$ \ell_{1}, \ldots, \ell_{k} \in \mathbb{F}[x_{0}, \ldots, x_{n}]_{1} $ and scalars $ (\lambda_{1}^{j}, 
\ldots, \lambda_{k}^{j}) \in \mathbb{F}^{k}-\{\underline{0}\} $, $ j \in \{1, \ldots, r\} $, such that
$$ f_{j} = \lambda_{1}^{j}\ell_{1}^{a_{j}}+ \ldots + \lambda_{k}^{j}\ell_{k}^{a_{j}} $$
for all $ j \in \{1, \ldots, r\} $, or, in vector notation,
\begin{equation}\label{eq:simWardec}
f = \sum_{i=1}^{k}(\lambda_{i}^{1}\ell_{i}^{a_{1}}, \ldots,\lambda_{i}^{r}\ell_{i}^{a_{r}}). 
\end{equation}
The minimum number of summands $ k $ in a simultaneous Waring decomposition of $ f $ over $ \mathbb{F} $ is called the 
\emph{rank} of $ f $ over $ \mathbb{F} $. In particular, any summand in (\ref{eq:simWardec}) has rank $ 1 $ over $ \mathbb{F} $.
\end{defn}

\begin{rem}
If $ f $ is a polynomial vector over $ \mathbb{R} $, then the rank of $ f $ over $ \mathbb{R} $ is greater or equal 
to the rank of $ f $ over $ \mathbb{C} $.
\end{rem}

\begin{defn}
If $ f $ admits a unique presentation (resp. several presentations) over $ \mathbb{F} $ as in (\ref{eq:simWardec}), then $ f $ is 
said to be \emph{identifiable} (resp. \emph{unidentifiable}) over $ \mathbb{F} $.
\end{defn}

Let us denote by $ \p^{n} = \p_{\c}^{n} $ the $ n $-dimensional complex projective space. 

\begin{rem}
According to Section $ 2.2 $ of \cite{AngeGaluppiMellaOtt18}, from a geometric point of view the variety of rank $ 1 $ polynomial vectors over 
$ \mathbb{C} $ with parameters $ (n,r;a_{1}, \ldots, a_{r}) $ is the projective bundle 
$$ X^{n}_{a_{1}, \ldots,a_{r}} = \p(\mathcal{O}_{\p^n}(a_{1}) \oplus \ldots \oplus \mathcal{O}_{\p^n}(a_{r})) \subset 
\p(H^{0}(\mathcal{O}_{\p^n}(a_{1})\oplus \ldots \oplus  \mathcal{O}_{\p^n}(a_{r}))) =  \p^{N-1} $$
where $ N = \sum_{j=1}^{r}\binom{a_{j}+n}{a_{j}} $. Being the projectivization of a vector bundle of rank $r$ over $\p^n$, then
$ X^{n}_{a_{1}, \ldots,a_{r}}$ has projective dimension $r+n-1$ (\cite{Hartshorne}, Ex. II.7.10). This immersion corresponds to the canonical invertible sheaf 
$ \o_{X^{n}_{a_{1}, \ldots,a_{r}}}(1) $. If a polynomial vector $ f \in \p^{N-1} $ can be written as in (\ref{eq:simWardec}), then 
$ f $ belongs to $ S^{k}(X^{n}_{a_{1}, \ldots,a_{r}}) $, the \emph{$ k $-secant variety} of $ X^{n}_{a_{1}, \ldots,a_{r}} $. 
We refer to \cite{Landsberg} for basics about secant varieties and the defectivity problem.
\end{rem}

According to \cite{AngeGaluppiMellaOtt18}, we give the following:

\begin{defn}\label{def:perfect}
The set of parameters $ (n,r;a_{1}, \ldots, a_{r}) $ yields a \emph{perfect case} if 
\begin{equation}\label{eq:perfect}
N = k(\dim(X^{n}_{a_{1}, \ldots,a_{r}})+1). 
\end{equation}
\end{defn}

\begin{rem}\label{rem:perfect}
Throughout the paper we assume to be in a perfect case, i.e. $ k $ is equal to the quotient $N/ (\dim(X^{n}_{a_{1}, \ldots,a_{r}})+1)$ 
and $ S^{k}(X^{n}_{a_{1}, \ldots,a_{r}}) $ is expected to fill the ambient space $ \p^{N-1} $. Moreover, as stated in \cite{CCi06}, 
the general fibre of the \emph{$ k $-secant map} 
$$ p^{k}_{X^{n}_{a_{1}, \ldots,a_{r}}}: S^{k}_{X^{n}_{a_{1}, \ldots,a_{r}}} \to S^{k}(X^{n}_{a_{1}, \ldots,a_{r}}) $$
where $ S^{k}_{X^{n}_{a_{1}, \ldots,a_{r}}} $ denotes the \emph{abstract k-secant variety} of $ X^{n}_{a_{1}, \ldots,a_{r}} $, 
is expected to be finite, that is there are only finitely many projective spaces of dimension $ k-1 $ which are $ k $-secant to 
$ X^{n}_{a_{1}, \ldots,a_{r}} $ and pass through the general polynomial vector $ f \in \p^{N-1}$. 
Therefore, in this case, we expect finitely many simultaneous Waring decompositions for $ f $.
\end{rem}

\begin{example}
In the classical Waring setting, i.e. when $ r=1 $, if, for simplicity, $ a_{1} = d $, then the variety under investigation is
$$ X_{d}^{n} = \p(\o_{\p^n}(d)) = \nu_{d}(\p^n) \subset \p^{\binom{d+n}{n}-1} $$ 
where $ \nu_{d}: \p^n \to \p^{\binom{d+n}{n}-1} $ is the Veronese embedding of $ \p^n $ of degree $ d $.
\end{example}

\begin{example}\label{rem:London}
Assume that $ n=2, \,r=3, a_{1}=a_{2}=a_{3}=3 $. We will refer to this case as {\it London's case}, because it was studied by London in
\cite{London90}. 

We focus on the $ 4 $-dimensional projective bundle
$$ X^{2}_{3,3,3} = \p(\mathcal{O}_{\p^2}(3)^{\oplus 3}) \subset \p(H^{0}(\mathcal{O}_{\p^2}(3)^{\oplus 3})) =  \p^{29}. $$
The condition (\ref{eq:perfect}) is satisfied by $ k=6 $. We note that 
\begin{equation}\label{eq:iso1}
X^{2}_{3,3,3} \cong \p(\mathcal{O}_{\p^2}(1)^{\oplus 3}\otimes \mathcal{O}_{\p^2}(2)) \cong  
\p(\mathcal{O}_{\p^2}^{\oplus 3} \otimes \mathcal{O}_{\p^2}(3)). 
\end{equation}
It is a standard fact that the projective bundle $ X_{1,1,1}^{2} = \p(\mathcal{O}_{\p^2}(1)^{\oplus 3}) $ is isomorphic to 
Seg$( \p^2 \times \p^2)\subset \p^8 $, where Seg denotes the Segre embedding (see e.g. Section 2.2 of \cite{AngeGaluppiMellaOtt18}). 
Namely, let us choose a $ \p^2 $ associated to one of the 
three $\mathcal{O}_{\p^2}(1)$'s and an element $ \phi \in $ Aut$ (\p^2) $ with no fixed points: if, for any $ P \in \p^2 $, we consider 
the projective plane $<P, \phi(P),\phi(\phi(P))>$, we get an image of $ \p^2 \times \p^2 $.\\
\indent By using (\ref{eq:iso1}), it turns out that $ X^{2}_{3,3,3} $ is the image of the embedding $ i_{3,3,3}^{2} $ of 
$ X_{1,1,1}^{2} $ in $ \p^{29} $ given by divisors of type $ (1,3) = (3,3) - (2,0) $ over Seg$( \p^2 \times \p^2) $. 
We observe that a divisor of type $(2,0)$ over Seg$( \p^2 \times \p^2) $ is defined by the Segre product of $ C \times \p^2 $, 
where $ C \subset \p^2 $ is a divisor of degree $ 2 $, i.e. a conic. Since $ C = \nu_{2}(\p^1) $ is the image of the quadratic 
Veronese embedding of $ \p^1 $, we have that Seg$( C \times \p^2) $ corresponds to an embedding of $ \p^1 \times \p^2 $, 
which is usually called the Segre-Veronese embedding. Thus $ i_{3,3,3}^{2} $ is given by cubic hypersurfaces of $ \p^8 $ 
containing Seg$ (C \times \p^2) $, where $ C \subset \p^2 $ is a chosen conic.
\end{example}

\begin{rem}\label{rem:nodef}
With the aid of an algorithm based on Terracini's lemma \cite{Terracini11} and described in Section $ 4 $ of \cite{AngeGaluppiMellaOtt18}, 
we know that the secant variety $ S^{6}(X^{2}_{3,3,3}) $ fills the ambient space $ \p^{29} $, as expected. Thus, according to Remark \ref{rem:perfect}, 
the general polynomial vector with parameters $ (n,r;a_{1}, \ldots, a_{r}) = (2,3;3,3,3) $ has finitely many simultaneous Waring decompositions with $ k = 6 $. 
\end{rem}

\begin{example}\label{rem:generalcase}
In the case considered in Corollary \ref{thm:collect2} we have that $ n=2, \,r\geq 4,\, a_{j}=2 $ for $ j \geq 4 $, $ a_{1}=a_{2}=a_{3}=3 $ and the projective 
bundle under investigation is 
$$ X^{2}_{3,3,3,2,\ldots,2} = \p(\mathcal{O}_{\p^2}(3)^{\oplus3} \oplus \mathcal{O}_{\p^2}(2)^{\oplus3}) \subset \p(H^{0}
(\mathcal{O}_{\p^2}(3)^{\oplus3} \oplus \mathcal{O}_{\p^2}(2)^{\oplus r-3})) =  \p^{29+6(r-3)}$$
which still gives a perfect case if $ k=6 $. We note that 
\begin{equation}\label{eq:iso}
X^{2}_{3,3,3,2,\ldots,2} \cong \p((\mathcal{O}_{\p^2}(1)^{\oplus3} \oplus \mathcal{O}_{\p^2}^{\oplus r-3})\otimes 
\mathcal{O}_{\p^2}(2)) \cong  \p((\mathcal{O}_{\p^2}^{\oplus3} \oplus \mathcal{O}_{\p^2}(-1)^{\oplus r-3}) \otimes 
\mathcal{O}_{\p^2}(3)). 
\end{equation}
As a consequence of the argument used in Example \ref{rem:London}, $ X_{1,1,1,0,\ldots,0}^{2} = \p(\mathcal{O}_{\p^2}(1)^{\oplus3} \oplus 
\mathcal{O}_{\p^2}^{\oplus r-3}) \subset \p^{r+5} $ in $ \p^{29+6(r-3)} $ can be described as the cone with vertex a linear space $ L \cong \p^{r-4} $ 
over the Segre variety 
Seg$( \p^2 \times \p^2) \cong  X_{1,1,1}^{2} \subset \p^8 $, that is $ X_{1,1,1,0,\ldots,0}^{2} = $ C$_{L}(X_{1,1,1}^{2})$. 
Indeed, the summand $ \mathcal{O}_{\p^2}^{\oplus r-3} $ gives the vertex $ L $ of the cone. \\
\indent By using (\ref{eq:iso}), $ X^{2}_{3,3,3,2, \ldots, 2} $ is the image of the embedding $ i^{2}_{3,3,3,2,\ldots,2} $ of $ X_{1,1,1,0,\ldots,0}^{2} $ in $ \p^{29+6(r-3)} $, 
defined by divisors of type $ (1,3) = (3,3) - (2,0) $ over C$_{L}($Seg$( \p^2 \times \p^2)) $, that is, extending the argument of Example \ref{rem:London}, 
by cubic hypersurfaces of $ \p^{r+5} $ (which cut divisors of type $ (3,3) $ over $ X_{1,1,1,0,\ldots,0}^{2}) $ containing the cone 
C$_{L}($Seg$(C \times \p^2)) $, where $ C \subset \p^2 $ is a conic.
\end{example}

\begin{example}\label{rem:new}
Assume that $ n=2, \,r=4, a_{1}=a_{2}=a_{3}=a_{4}=4 $. 

We focus on the $5$-dimensional projective bundle
$$ X^{2}_{4,4,4,4} = \p(\mathcal{O}_{\p^2}(4)^{\oplus 4}) \subset \p(H^{0}(\mathcal{O}_{\p^2}(4)^{\oplus 4})) =  \p^{59}. $$
The condition (\ref{eq:perfect}) is satisfied by $ k=10 $. We note that 
\begin{equation}\label{eq:isonew}
X^{2}_{4,4,4,4} \cong \p(\mathcal{O}_{\p^2}^{\oplus 4} \otimes \mathcal{O}_{\p^2}(4)). 
\end{equation}
As in Example \ref{rem:London} or by Section 2.2 of \cite{AngeGaluppiMellaOtt18}, the projective bundle $ X_{1,1,1,1}^{2} = \p(\mathcal{O}_{\p^2}(1)^{\oplus 4}) $ 
is isomorphic to  Seg$( \p^3 \times \p^2)\subset \p^{11} $.

By using (\ref{eq:isonew}), it turns out that $ X^{2}_{4,4,4,4} $ is the image of the embedding $ i_{4,4,4,4}^{2} $ of 
$ X_{1,1,1,1}^{2} $ in $ \p^{59} $ given by divisors of type $ (1,4) $ over Seg$( \p^3 \times \p^2) $. 
\end{example}

\begin{rem}\label{rem:nodenew}
With the aid of an algorithm based on Terracini's lemma \cite{Terracini11} and described in Section $ 4 $ of \cite{AngeGaluppiMellaOtt18}, 
we know that the secant variety $ S^{10}(X^{2}_{4,4,4,4}) $ fills the ambient space $ \p^{59} $, as expected. Thus, according to Remark \ref{rem:perfect}, 
the general polynomial vector with parameters $ (n,r;a_{1}, \ldots, a_{r}) = (2,4;4,4,4,4) $ has finitely many simultaneous Waring decompositions with $ k = 10 $. 
\end{rem}

We end this introductory section with the definition of \emph{catalecticant map} associated to polynomial vectors, according to \cite{IK}.

\begin{defn}\label{catal} Let $ f = (f_{1},\dots,f_{r}) $ be a general polynomial vector such that $ f_{j} \in \c[x_0,x_1,x_2]_{d} $, for $ j=1,\dots, r $.
For any $h$, identify the dual space $  (\sym^{h}\c^{3})^{\vee} $ as the space of polynomial derivatives in $ 3 $ variables of order $h$. 
The map 
$$ C_{f}^h: (\sym^{h}\c^{3})^{\vee} \to (\sym^{d-h}\c^{3})^{\oplus r} $$
given by the contraction by $ f $ is called the \emph{catalecticant map} of order $h$ associated to $f$.   \\
When $h=d-1$, we will write simply $C_f$ instead of $C^{d-1}_f$.
\end{defn}  

\begin{rem}\label{remcatal}
$ C_{f}^h $ is a linear map. The associated matrix, which, by abuse, we will denote again by $ C_{f}^h $,
 has $ 0 $-degree entries and is divided into $r $ blocks, i.e.: 
\begin{equation}\label{eq:CfL}
C_{f}^h = \begin{pmatrix} DD_1 \\
\cdots \\
DD_r \cr
\end{pmatrix} 
\end{equation}
where $ DD_{j} $ is a $\binom{d-h+2}2\times \binom{h+2}2$ block such that in each column there are the coefficients
of  one  derivative of order $h$ of $f_{j} $.
\end{rem}

As explained in the Introduction, every catalecticant map provides good candidates for the confinement of the decompositions of a polynomial vector,
at least when it does not surject. \\
In the sequel, we will use only catalecticant maps of order $d-1$, for vectors of forms of degree $d$, which will be sufficient for our purposes.

\section{Elliptic normal curves}\label{sec:ellcurv}

An \emph{elliptic normal curve} $ \mathcal{C} \subset \p^{m} $ is an irreducible, smooth curve of genus $ 1 $ 
and degree $ m+1 $ that is not degenerate. 
As for any curve (see Remark 3.1 of \cite{CCi02a}), for all $ h \in \mathbb{N} $, the $h$-secant variety $ S^{h}(\mathcal{C}) $ of $ \mathcal{C} $ 
has the expected dimension  $ \dim S^{h}(\mathcal{C}) = \min \{2h-1,m\} $.

Recall that the \emph{rank} of a point $ P \in \langle\mathcal{C}\rangle \, = \p^{m} $ over $ \c $ \emph{with respect to $ \mathcal{C} $} 
is the minimal number of points of $ \mathcal{C} $ needed to generate a projective space over $ \c $ containing $ P $.

\begin{rem}
If $ \mathcal{C} \subset \p^{2k-1} $ is an elliptic normal curve of even degree $ m+1=2k $ and $ P \in \langle\mathcal{C}\rangle \, = \p^{2k-1} $ 
is a general point, then the rank of $ P $ over $ \c $ with respect to $ \mathcal{C} $ is $ k $. Therefore the value  $ k $ 
corresponds to the generic rank and provides a perfect case in the sense of Definition 
\ref{def:perfect} and Remark \ref{rem:perfect}.
\end{rem}

In the specific case of elliptic curves, one knows the number of decompositions of a general tensor of
generic rank. Namely, the following holds:

\begin{prop}[Chiantini-Ciliberto, 2006, \cite{CCi06}]\label{prop:elliptic}
The number of $ k $-secant $ (k-1) $-spaces to an elliptic normal curve $ \mathcal{C} $ of degree $ 2k $ in $ \p^{2k-1} $, 
passing through the general point of $ \p^{2k-1} $, is $ 2 $.
\end{prop}

With the notation of Example \ref{rem:London} we prove the following:

\begin{prop}\label{prop:sexticellLondon}
The intersection of  $X^{2}_{1,1,1} =\Seg(\p^2\times\p^2)$ with a general linear space of dimension $5$ in $\p^8$ is an irreducible normal elliptic
curve of degree $6$. Thus, for a general choice of $6$ points
 $ \tilde{p}_{1}, \ldots, \tilde{p}_{6} \in  X_{1,1,1}^{2} $ there exists a unique (sextic) elliptic normal 
curve $ \mathcal{C} $ of $ \p^5 $ passing through the points.
\end{prop}

\begin{proof}
$X^{2}_{1,1,1}=\Seg(\p^2\times\p^2)$ is smooth, irreducible and non-degenerate, of degree $6$. Thus the intersection $\mathcal C$ of $X^{2}_{1,1,1}$
with a general linear space of dimension $5$ is a smooth irreducible and non-degenerate curve of degree $6$ in $\p^5$ 
(apply consecutively the classical Bertini's Theorems, see \cite{Hartshorne} Theorem II.8.18 and Remark III.7.9.1).
The fact that $\mathcal C$ has genus $1$ has been classically proved by Scorza in \cite{Scorza08} (the classical notion of \emph{variety with elliptic curve sections}
means exactly that a general curve section of Seg$(\p^2\times\p^2)$ has genus $1$).

The second claim follows because $6$ general points $\tilde{p}_{1}, \ldots, \tilde{p}_{6}$ of $X^{2}_{1,1,1} $ span a general  $\p^5$ in $ \p^8 $,
as $X^{2}_{1,1,1} $ is non-degenerate.

Notice that the uniqueness of the elliptic normal curve $\mathcal C$ of degree $6$ through $ \tilde{p}_{1}, \ldots, \tilde{p}_{6} $
follows by Bézout Theorem, since Seg$(\p^2\times\p^2)$ has degree $6$, thus it cannot intersect a general $\p^5$, span of the $6$ general points, in a curve
of degree bigger than $6$.
\end{proof}

\begin{rem} A confirmation of the statement of Proposition \ref{prop:sexticellLondon} can be obtained also from a computational point of view, 
via the software system Macaulay2 \cite{Macaulay2} (for more details on the script see the ancillary file {\tt{CatIntConf.pdf}}).
\end{rem}

By extending the construction in the proof of Proposition \ref{prop:sexticellLondon} to cones, we obtain the following:

\begin{prop}\label{prop:sexticell}
Let $ p_{1}, \ldots, p_{6} \in  X_{1,1,1,0, \ldots, 0}^{2} $ be general points. There exists a unique (sextic) elliptic normal curve $ \mathcal{C}_{1} $ 
of $ \p^5 $ passing through $ p_{1}, \ldots, p_{6} $.
\end{prop}

\begin{proof}
Since, according to Example \ref{rem:generalcase}, $  X_{1,1,1,0, \ldots, 0}^{2} =$ C$_{L}(X_{1,1,1}^{2}) $, then exactly the same proof given in 
Proposition \ref{prop:sexticellLondon} for $ X_{1,1,1}^{2} $  works. 
\end{proof}

As a consequence of Example \ref{rem:London} and Proposition \ref{prop:sexticellLondon}, we get the following:

\begin{prop}\label{prop:ellcurveLondon}
 Let $ \tilde{q}_{1}, \ldots, \tilde{q}_{6} \in  X^{2}_{3,3,3} $ be general points. There exists a unique 
elliptic normal curve 
$ \mathcal{C}_{2} $ of $ \p^{11} $ passing through $ \tilde{q}_{1}, \ldots, \tilde{q}_{6} $.
\end{prop}
\begin{proof}
Let $ \tilde{p}_{1}, \ldots, \tilde{p}_{6} \in X_{1,1,1}^{2} $ such that $ \tilde{q}_{j} = i^{2}_{3,3,3}(\tilde{p}_{j}) $, 
$ j \in \{1, \ldots, 6\} $ and let $ \mathcal{C} $ be the sextic elliptic normal curve corresponding to $ \tilde{p}_{1}, \ldots, \tilde{p}_{6} $ 
in the sense of Proposition \ref{prop:sexticellLondon}. We claim that $ \mathcal{C}_{2} = i^{2}_{3,3,3}(\mathcal{C}) $ 
is an elliptic normal curve of degree $ 12 = 18-6 $. Namely, $\mathcal C_2$ is smooth irreducible of genus $1$ because it is isomorphic
to $\mathcal C$. Moreover, since $ i^{2}_{3,3,3}$ is defined by means of cubics and 
$\deg \mathcal{C} = 6 $, and since the base locus of  $i^{2}_{3,3,3}$ is Seg$(C\times\p^2)$, where $C$ is a conic curve (see Example \ref{rem:generalcase}),
then the degree of $ \mathcal{C}_{2}$ is $ 18 $ minus the degree of the intersection 
Seg$(C \times \p^2) $ $ \cap  \langle\mathcal{C}\rangle $. Notice that the intersection is $0$-dimensional, since the points are general and 
Seg$(C \times \p^2) $ is a $ 4 $-fold in $ \p^8 $ while $ \langle \mathcal{C} \rangle = \langle \tilde{p}_{1}, \ldots, \tilde{p}_{6}\rangle \cong \p^5 $. 
Since $\deg($Seg$(C \times \p^2)) = \deg (C \times \p^2) = 2\cdot \deg (\p^1 \times \p^2) = 6 $, we get that 
$ \deg \mathcal{C}_{2} = 18-6 = 12$, as claimed. We remark that $ \langle \mathcal{C}_{2} \rangle  = \p^{11} $: indeed the dimension 
of the projective space spanned by $ \mathcal{C}_{2} $ can't exceed $ 11 $ by the Riemann-Roch Theorem (see \cite{Hartshorne}, 
Theorem IV.1.3) and it can't be lower 
than $ 11 $, otherwise, through the general point of $ \langle \mathcal{C}_{2} \rangle $ would pass infinitely many $ 6 $-secant spaces, 
which implies that, by Theorem 2.5 of \cite{CCi06}, the secant variety $ S^{6}(X^{2}_{3,3,3}) $ is defective and 
this is not possible as explained in Remark \ref{rem:nodef}.
\end{proof}

\section{The case of three ternary cubics }\label{sec:cat}

Let $ f = (f_{1},f_{2},f_{3}) $ be a general polynomial vector such that $ f_{j} \in \c[x_0,x_1,x_2]_{3} $, for $ j \in \{1,2,3\} $.\\
Consider the catalecticant map $C_f$, defined by derivatives of order $2$, associated to $f$ (see Definition \ref{catal})
$$ C_{f}: (\sym^{2}\c^{3})^{\vee} \to (\sym^{1}\c^{3})^{\oplus 3}. $$
As pointed out in Remark \ref{remcatal}, the associated $ 9 \times 6 $ matrix (that we call again $ C_{f} $)
is divided into $ 3 $ blocks of type $3\times 6$:
\begin{equation}
C_{f} = \begin{pmatrix} DD_1 \\
DD_2 \\ 
DD_3 \cr
\end{pmatrix} 
\end{equation}

We have the following:

\begin{prop}\label{prop:propertiesL}
For a generic choice of $ f $ with $ (n,r;a_{1}, \ldots, a_{r})=(2,3;3,3,3) $ we have that:\\
$(i)$ $C_{f}$ is an injective map;\\
$(ii)$ $ \p(\im(C_{f})) $ intersects $ X_{1,1,1}^{2} $ in a sextic elliptic normal curve.
\end{prop}

\begin{proof} 
Fix a Waring decomposition of $ f \in \p^{29} $ given by the points $ \tilde{q}_{i} =  (\lambda_{i}^{1}\ell_{i}^3,
\lambda_{i}^{2}\ell_{i}^3,\lambda_{i}^{3}\ell_{i}^3) \in X^{2}_{3,3,3}$, where we can assume, without loss of generality, 
that $ \ell_{i} = x_{0}+\nu_{i}^{1}x_{1}+\nu_{i}^{2}x_{2} $ for $ i \in \{1, \ldots, 6\} $. \\
Immediate direct computations show that 
$$ C_{f} = \begin{pmatrix} 
\displaystyle\sum_{i=1}^{6} C_{i} & \displaystyle\sum_{i=1}^{6} \nu_{i}^{1}C_{i} & \displaystyle\sum_{i=1}^{6} 
\nu_{i}^{2}C_{i} & \displaystyle\sum_{i=1}^{6} \nu_{i}^{1}\nu_{i}^{1}C_{i} & \displaystyle\sum_{i=1}^{6} \nu_{i}^{1}
\nu_{i}^{2}C_{i} & \displaystyle\sum_{i=1}^{6} \nu_{i}^{2}\nu_{i}^{2}C_{i} \cr
\end{pmatrix}
$$
where 
\begin{equation}\label{eq:Ci}
C_{i} = 6 \begin{pmatrix} \lambda_{i}^{1} & \nu_{i}^{1}\lambda_{i}^{1} & \nu_{i}^{2}\lambda_{i}^{1} & \lambda_{i}^{2} & 
\nu_{i}^{1}\lambda_{i}^{2} & \nu_{i}^{2}\lambda_{i}^{2} & \lambda_{i}^{3} & \nu_{i}^{1}\lambda_{i}^{3} 
& \nu_{i}^{2}\lambda_{i}^{3} \cr
\end{pmatrix}^{t}, \, i \in \{1, \ldots, 6\}. 
\end{equation}
In order to work with general $ f $, by semicontinuity we choose to assign random values to the $ \lambda_{i}^{j}$'s and $ \nu_{i}^{h} $'s, 
$ h \in \{1,2\} $, $ i \in \{1, \ldots, 6\} $, $ j \in \{1,2,3\} $ and we compute the corresponding $ f_{1},f_{2}, f_{3} $. 
Then we construct the matrix $ C_{f} $: if its rank is $ 6 $, then we can conclude that (i) holds. This fact can be verified 
from a computational point of view via Macaulay2 (for more details on the script see the ancillary file {\tt{CatIntConf.pdf}}). 
As a consequence, $ \p(\im(C_{f})) \cong \p^5 \subset  \p((\sym^{1}\c^{3})^{\oplus 3}) \cong \p^{8} $ and, 
since $ X_{1,1,1}^{2} \subset \p^{8} $ is $ 4 $-dimensional, then $ \dim(\p(\im(C_{f})) \cap X_{1,1,1}^{2}) = 1 $. 
By arguing as in the proof of Proposition \ref{prop:sexticellLondon} we get that the  the intersection curve has degree $ 6 $, 
as one can compute directly (see the instructions added to the previous script and provided in the ancillary file {\tt{CatIntConf.pdf}}). Thus (ii) follows.
\end{proof}

Therefore we have the following:

\begin{prop}\label{prop:addcat}
Let $ f $ be a general polynomial vector with $ (n,r;a_{1}, \ldots, a_{r})=(2,3;3,3,3) $ and let $\tilde{q}_{i} =  (\lambda_{i}^{1}
\ell_{i}^3,\lambda_{i}^{2}\ell_{i}^3,\lambda_{i}^{3}\ell_{i}^3) \in X^{2}_{3,3,3}$, $ i \in \{1, \ldots, 6\} $, be the points of a 
simultaneous Waring decomposition of $ f $. Then 
$$ \p(\im C_{f}) =\, \langle [\ell_{1}^{\vee} \otimes (\lambda_{1}^{1}, \lambda_{1}^{2},\lambda_{1}^{3})], \ldots, [\ell_{6}^{\vee}
 \otimes (\lambda_{6}^{1},\lambda_{6}^{2},\lambda_{6}^{3})] \rangle $$
where $ \ell_{i}^{\vee} $ denotes a representative vector of the point corresponding to $ \ell_{i} $ in the dual projective plane 
$ (\p^2)^{\vee} $ and $ [\ell_{i}^{\vee} \otimes (\lambda_{i}^{1},\lambda_{i}^{2},\lambda_{i}^{3})] $ the equivalence class 
of $ \ell_{i}^{\vee} \otimes (\lambda_{i}^{1},\lambda_{i}^{2},\lambda_{i}^{3}) $ in $ \p^8 $, for $ i \in \{1, \ldots, 6\} $.
\end{prop}

\begin{proof}
Without loss of generality, we can assume that $ \ell_{i} = x_{0}+\nu_{i}^{1}x_{1}+\nu_{i}^{2}x_{2}$, 
for $ i \in \{1, \ldots, 6\} $. By means of equation (\ref{eq:Ci}), we get that 
$$ [C_{i}] = [(1,\nu_{i}^{1},\nu_{i}^{2})\otimes(\lambda_{i}^{1},\lambda_{i}^{2},\lambda_{i}^{3})] = [\ell_{i}^{\vee}
\otimes(\lambda_{i}^{1},\lambda_{i}^{2},\lambda_{i}^{3})], \,  i \in \{1, \ldots, 6\}. $$
Thus, $ \p(\im C_{f}) \subset \langle [\ell_{1}^{\vee} \otimes (\lambda_{1}^{1}, \lambda_{1}^{2},\lambda_{1}^{3})], \ldots, 
[\ell_{6}^{\vee} \otimes (\lambda_{6}^{1},\lambda_{6}^{2},\lambda_{6}^{3})] \rangle $.
Since, by Proposition \ref{prop:propertiesL} part (i), the rank of $ C_{f} $ is generically $ 6 $, the previous inclusion is actually an equality, as desired.\\
\noindent 
A computational verification of this fact can be done via Macaulay2. For more details see the ancillary file {\tt{CatIntConf.pdf}}. 
\end{proof}

From Proposition \ref{prop:addcat} and Proposition \ref{prop:ellcurveLondon}, we get the following:

\begin{cor}\label{cor:C2}
For all $ i \in \{1, \ldots, 6\} $, $ [\ell_{i}^{\vee} \otimes (\lambda_{i}^{1},\lambda_{i}^{2},\lambda_{i}^{3})] $ belongs 
to the sextic elliptic normal curve $ \mathcal{C} = \p(\im C_{f}) \cap X_{1,1,1}^{2} $ and $\tilde{q}_{i} =  (\lambda_{i}^{1}
\ell_{i}^3,\lambda_{i}^{2}\ell_{i}^3,\lambda_{i}^{3}\ell_{i}^3) $ lies on the elliptic normal curve of degree $ 12 $ $ 
\mathcal{C}_{2} = i_{3,3,3}^{2}(\mathcal{C}) $. 
\end{cor}

As a consequence we get:

\begin{thm}\label{thm:collect1}
The general $ f = (f_{1},f_{2},f_{3}) $ such that $ f_{i} \in \c[x_0,x_1,x_2]_{3} $, for $ i \in \{1,2,3\} $, has two simultaneous  
Waring decompositions with $ k=6 $ summands.
\end{thm}

\begin{proof}
Fix a Waring decomposition of $ f \in \p^{29} $ given by the points $ \tilde{q}_{i} = (\lambda_{i}^{1}\ell_{i}^3,\lambda_{i}^{2}
\ell_{i}^3,\lambda_{i}^{3}\ell_{i}^3) \in X^{2}_{3,3,3}$ for $ i \in \{1, \ldots, 6\} $. Thanks to Corollary \ref{cor:C2}, we get that $ \tilde{q}_{1}, 
\ldots, \tilde{q}_{6}$ belong to an elliptic normal curve $\mathcal{C}_{2}$ of degree $ 12 $ in $ X^{2}_{3,3,3} $, which spans the $ \p^{11} $ 
where $ f $ lies. The two $ 6 $-secant spaces to $ \mathcal{C}_{2} $ containing $ f $, that exist because of Proposition \ref{prop:elliptic}, 
determine the two required Waring decompositions of $ f $. 
\end{proof}

\begin{rem}
According to Example \ref{rem:London} we have $ X^{2}_{3,3,3} \cong $ Seg$(\p^{2} \times \nu_{3}(\p^{2})) $. Then Theorem 
\ref{thm:collect1} implies that the Veronese surface $ \nu_{3}(\p^{2}) $ is not \emph{$ (2,6) $-identifiable}, in the sense of 
\cite{BallBernCatalC13}. In particular, the general $ \p^{2} \subset <P_{1}, \ldots, P_{6}> $, where $ P_{1}, \ldots, P_{6} \in \nu_{3}(\p^{2}) $ 
are independent points, is contained in one more $ \p^{5} $ $ 6 $-secant to $ \nu_{3}(\p^{2}) $, besides $ <P_{1}, \ldots, P_{6}> $. 
Therefore, the general linear system of ternary cubics $ \mathcal{E} \subset  \p(H^{0}(\mathcal{O}_{\p^2}(3)^{\oplus 3})) $ of 
dimension $ 2 $ and rank $ 6 $ is computed by exactly $ 2 $ sets of polynomial vectors of rank $ 1 $.
\end{rem}

By combining Theorem \ref{thm:collect1} with the observation that the linear space of quadratic forms in three variables has dimension $6$, we get:

\begin{cor}\label{thm:collect2} 
For any $ r \geq 4 $, the general $ f = (f_{1},f_{2},f_{3},f_{4},\ldots,f_{r}) $ such that $ f_{i} \in \c[x_0,x_1,x_2]_{3} $, $ i 
\in \{1,2,3\} $, and $ f_{i} \in \c[x_0,x_1,x_2]_{2} $, $i \in \{4, \ldots,r\} $, has two Waring decompositions with $ k=6 $ summands.

In particular for $r=4$, the general $ f = (f_{1},f_{2},f_{3},f_{4}) $ such that $ f_{i} \in \c[x_0,x_1,x_2]_{3} $, for $ i \in \{1,2,3\} $, and $ f_{4} \in \c[x_0,x_1,x_2]_{2} $  
has two simultaneous Waring decompositions with $ k=6 $ summands.
\end{cor}

\begin{proof}
It is immediate since the decompositions of a general vector $(f_{1},f_{2},f_{3})$ of cubics is given by powers of $6$ general linear forms,
whose squares generate the space of all conics. 
\end{proof}

As we anticipated in section \ref{sec:intr}, Corollary \ref{thm:collect2} yields a proof for a conjecture stated in \cite{AngeGaluppiMellaOtt18}. 

\begin{rem}\label{rem:extcat}
We notice that the arguments used to prove Theorem \ref{thm:collect1} extend in a natural way to get an alternative proof of Corollary \ref{thm:collect2}. 
The main difference between the two cases is the role of cones, as stressed in Example \ref{rem:generalcase}. Indeed, fix a Waring decomposition of 
$ f \in \p^{29+6(r-3)} $ given by $ q_{i} = (\lambda_{i}^{1}\ell_{i}^3,\lambda_{i}^{2}\ell_{i}^3,\lambda_{i}^{3}\ell_{i}^3,\lambda_{i}^{4}\ell_{i}^2,...,
\lambda_{i}^{r}\ell_{i}^2) $ 
$\in X^{2}_{3,3,3,2,\ldots,2}$, $ i \in \{1, \ldots, 6\} $. By generalizing Corollary \ref{cor:C2}, 
we can show that the \emph{catalecticant map} 
$$ C_{f}: (\sym^{2}\c^{3})^{\vee} \to (\sym^{1}\c^{3})^{\oplus 3} \oplus \c ^{\oplus r-3} $$ 
given by the contraction by $ f $, is such that $  q_{i} $ belongs to the elliptic normal curve of degree $ 12 $
$$ \mathcal{C}_{3} =  i^{2}_{3,3,3,2,\ldots,2}(\p(\im C_{f}) \cap X_{1,1,1,0,\ldots,0}^{2}) $$
and $ f \in \langle \mathcal{C}_{3} \rangle = \p^{11} $. Thus, Proposition \ref{prop:elliptic} allows us to conclude the proof.
\end{rem}

\section{Threefold sections of Seg$(\p^3\times\p^2)$}\label{sec:p3p2}

In this section we consider $X^2_{1,1,1,1}=\Seg(\p^3\times\p^2)$ and our aim is to describe a general threefold section $W$ of $X^2_{1,1,1,1}$.

We prove that $W$ is rational and we describe a birational map $\rho: \p^1\times\p^2 \dasharrow W$. 

\begin{prop} \label{3section} A general threefold section $W$ of  $X^2_{1,1,1,1}=\Seg(\p^3\times\p^2)$ is the image of $\p^1\times \p^2$ 
defined by divisors of type $(1,2)$ passing through $2$ generic points.
\end{prop}
\begin{proof} We need to define maps $\rho_1:\p^1\times \p^2\to \p^2$ and $\rho_2:\p^1\times \p^2\to \p^3$. The map
$\rho_1$ is the straightforward projection to the second factor, given by divisors of type $(0,1)$. The map $\rho_2$ is the composition
of the Segre embedding in $\p^5$ with a general internal projection from two points of $\Seg(\p^1\times\p^2)$, i.e. the map
given by divisors of type $(1,1)$ passing through two general points of $\p^1\times \p^2$. Since $\Seg(\p^1\times\p^2)$
is a threefold of degree $3$ in $\p^5$, then $\rho_2$ is birational. The sum $\rho=\rho_1+\rho_2$ is the required map.
The fact that the image is contained in a $3$-dimensional plane in $\p^{11}$ is a straigthforward computation.

By combining $\rho$ with a change of coordinates and counting parameters, we obtain that a general threefold section
of $\Seg(\p^3\times\p^2)$ is the image of a map like $\rho$.
\end{proof}

By collecting Proposition \ref{3section} and Example \ref{rem:new}, we get:

\begin{cor}\label{cor:i4W} The image of a  general threefold section $W$ of  $X^2_{1,1,1,1}=\Seg(\p^3\times\p^2)$ in the map $i^2_{4,4,4,4}$ 
is the image of $\p^1\times \p^2$ 
in the map given by divisors of type $(1,5)$ passing through $2$ generic points.
\end{cor}

\section{The case of four ternary quartics }\label{sec:cat4}

Let $ f = (f_{1},f_{2},f_{3}, f_{4}) $ be a general polynomial vector such that $ f_{j} \in \c[x_0,x_1,x_2]_{4} $, for $ j \in \{1,2,3,4\} $.
Consider the catalecticant map $C_f$, given by derivatives of order $3$, associated to $f$ (see Definition \ref{catal})
$$ C_{f}: (\sym^{3}\c^{3})^{\vee} \to (\sym^{1}\c^{3})^{\oplus 4}. $$
As pointed out in Remark \ref{remcatal}, the associated $ 12\times 10 $ matrix $ C_{f} $
is divided into $ 4 $ blocks of order $3\times 10$ 
\begin{equation}
C_{f} = \begin{pmatrix} DD_1 \\
DD_2 \\
DD_3\\
DD_4 \cr
\end{pmatrix}.
\end{equation}

We have the following:

\begin{prop}\label{prop:propertiesLnew}
For a generic choice of $ f $ with $ (n,r,a_{1}, \ldots, a_{r})=(2,4,4,4,4,4) $ we have that:\\
$(i)$ $C_{f}$ is an injective map;\\
$(ii)$ $ \p(\im(C_{f})) $ intersects $ X_{1,1,1,1}^{2} $ in a general section of codimension $2$, i.e. in a general threefold section.
\end{prop}

\begin{proof} 
The first claim follows from a computation which is absolutely equivalent with the one in the proof of Proposition \ref{prop:propertiesL}. 
(See also the ancillary file to this paper {\tt CatIntConf.pdf}.)

The second claim follows by counting parameters.
\end{proof}

Therefore we have the following:

\begin{cor}\label{cor:C2new}
Let $ f $ be a general polynomial vector with $ (n,r;a_{1}, \ldots, a_{r})=(2,4;4,4,4,4) $ and let $\tilde{q}_{i} =  (\lambda_{i}^{1}
\ell_{i}^4,\lambda_{i}^{2}\ell_{i}^4,\lambda_{i}^{3}\ell_{i}^4,\lambda_{i}^{4}\ell_{i}^4) \in X^{2}_{4,4,4,4}$, $ i \in \{1, \ldots, 10\} $, be the points of a 
simultaneous Waring decomposition of $ f $. 

Then for all $ i \in \{1, \ldots, 10\} $, $ [\ell_{i}^{\vee} \otimes (\lambda_{i}^{1},\lambda_{i}^{2},\lambda_{i}^{3},\lambda_{i}^{4})] $ belongs 
to a general threefold section $ W = \p(\im C_{f}) \cap X_{1,1,1,1}^{2} $ and $\tilde{q}_{i} =  (\lambda_{i}^{1}
\ell_{i}^4,\lambda_{i}^{2}\ell_{i}^4,\lambda_{i}^{3}\ell_{i}^4,\lambda_{i}^{4}\ell_{i}^4) $ lies on the image 
$ i_{4,4,4,4}^{2}(W) $. 
\end{cor}

By collecting Corollary \ref{cor:C2new} and Corollary \ref{cor:i4W} we immediately get the following:
\begin{cor}\label{cor:reduction} The number of simultaneous Waring decompositions with $ 10 $ summands of the general polynomial vector $ f $ 
such that $ (n,r;a_{1}, \ldots, a_{r})=(2,4;4,4,4,4) $ equals the number of presentations with $ 10 $ summands of the general $ \tilde{f} \in \p^{39} $ 
($ \tilde{f} $ is a polynomial of bi-degree $ (1,5) $ vanishing at $ 2 $ fixed generic points $ P,Q  \in \p^1\times \p^2$) of type 
\begin{equation}\label{eq:newdec0}
\tilde{f} = \sum_{i=1}^{10}g_{i}h_{i}^{5}
\end{equation}
where $ g_{i} \in \C[s,t]_{1} $, $ h_{i} \in \C[y_{0}, y_{1}, y_{2}]_{1} $ and $ g_{i}h_{i}^{5} $ vanishes at $ P,Q $, for any $ i \in \{1,\ldots, 5\} $. 
\end{cor}

Therefore our starting problem reduces to counting decompositions in the second case introduced in Corollary \ref{cor:reduction}. 
In order to do that, by means of  Bertini and Matlab software, we developed a procedure, which, inspired to the one introduced in the papers 
\cite{AngeGaluppiMellaOtt18} and \cite{Ange19}, allows us to prove from a computational point of view the following:

\begin{prop*}\label{prop:decred}
The general $ \tilde{f} \in \p^{39} $ admits $ 18 $ presentations with $ 10 $ summands as in (\ref{eq:newdec0}), up to re-ordering and re-scaling. 
\end{prop*}
\begin{proof}
Let $ \tilde{f} \in \p^{39} $  be a general element and let $ \{s,t\} $ (respectively $ \{y_{0},y_{1},y_{2}\} $) be a set of homogeneous coordinates for $ \p^{1} $ 
(respectively for $ \p^{2} $). 
Without loss of generality, we choose $ P = [1,0] \times [0,0,1] $ and $ Q = [0,1] \times [0,1,0] $ as generic points in $ \p^{1} \times \p^{2} $. If we write the polynomials 
 in \eqref{eq:newdec0} as $g_{i+1} = v_{4i+2}s+v_{4i+3}t$ and $h_{i+1}=y_0+v_{4i}y_1+v_{4i+1}y_2$, for $i=0,\dots,9$, then the condition $\tilde f(P)=\tilde f(Q)=0$, that is 
$$\sum_{i=1}^{10}g_i(1,0)h_i(0,0,1)=\sum_{i=1}^{10}g_i(0,1)h_i(1,0,0)=0$$
implies
$$\sum_{i=0}^9v_{4i+2}v_{4i+1} = \sum_{i=0}^9v_{4i+3}v_{4i} = 0.$$
Our aim is thus to find the set of $ \{v_{0}, \ldots, v_{39}\} \in \C^{40} $ such that $ \tilde{f} $ expresses as 
\begin{equation}\label{eq:newdec}
\tilde{f} = \sum _{i=0}^{9}[(v_{4i+2}s+v_{4i+3}t)(y_{0}+v_{4i}y_{1}+v_{4i+1}y_{2})^{5} - v_{4i+2}v_{4i+1}sy_{2}^5 - v_{4i+3}v_{4i} t y_{1}^5],
\end{equation}
where, in each summand, the monomials $ v_{4i+2}v_{4i+1}sy_{2}^5, \, v_{4i+3}v_{4i} t y_{1}^5 $ are subtracted so that $ \tilde{f} $ vanishes at $ P $ and $ Q $.
Equivalently, if $ [p_{0}, \ldots, p_{39}] $ represents a set of homogeneous coordinates for $ \tilde{f} $, then
\begin{equation} \label{eq:matrices}
\begin{pmatrix} 
p_{0} & \ldots & p_{14} & p_{15} & \ldots & p_{19} & 0 \cr
p_{20} & \ldots & p_{34} & 0 & p_{35} & \ldots & p_{39} \\
\end{pmatrix} =  \end{equation}
$$
= \sum_{i=0}^{9} \left[\begin{pmatrix} v_{4i+2} \cr v_{4i+3} \\ \end{pmatrix} \nu_{5}(1,v_{4i},v_{4i+1}) - \begin{pmatrix} 
0 & \ldots & 0 & 0 & \ldots & 0 & v_{4i+2}v_{4i+1} \cr
0 & \ldots & 0 & v_{4i+3}v_{4i} & 0 & \ldots & 0 \\
\end{pmatrix} \right]
$$
where $ \nu_{5} $ denotes the Veronese map of $ \p^{2} $ of degree $ 5 $.   The equality of matrices appearing in (\ref{eq:matrices}) translates into a 
non linear system consisting, in principle, of $ 42 $ equations. It is immediate to see that the condition $\tilde f(P)=\tilde f(Q)=0$ implies that two equations,
corresponding respectively to the monomials $sy_2^5$ and $ty_1^5$, are trivial. 
Thus, we get a square non linear system with $40$ equations, with $ \{v_{0}, \ldots, v_{39}\} $ as unknowns. We focus on the system $ F_{(p_{0}, \ldots, p_{39})}
([v_{0}, \ldots, v_{3}], \ldots, [v_{36}, \ldots, v_{39}]) $ obtained from the one mentioned above by re-arranging all the terms in each equation on one side of the equal sign. \\
In practice, to work with a general $ \tilde{f} $, we change $ \{v_{0}, \ldots, v_{39}\} $ with random complex numbers 
$ \{\underline{v}_{0}, \ldots, \underline{v}_{39}\} $ (\emph{start-point}) and then we compute, by means of (\ref{eq:matrices}), the corresponding 
$ [\underline{p}_{0}, \ldots, \underline{p}_{39} ] $ (vector of \emph{start-parameters}). By construction, the start-point is a solution of 
$ F_{(\underline{p}_{0}, \ldots, \underline{p}_{39})} $. In particular, the chosen start-point is the following:
{\small{$$ [\underline{v}_{0},\underline{v}_{1},\underline{v}_{2},\underline{v}_{3}] = [3.803150504548735 \cdot 10^{-1} + i  1.080968803617349 \cdot 10^{-1}, $$
$$ \quad\quad\quad\quad\quad\quad\quad\,\,  4.012914786260265 \cdot 10^{-2} + i 1.194906105430308 \cdot 10^{-2}, $$
$$ \quad\quad\quad\quad\quad\quad\quad\,\,-5.791791173087690\cdot 10^{-1} + i  7.036742613968909\cdot 10^{-1}, $$  
$$ \quad\quad\quad\quad\quad\quad\quad\,\,-3.976968438023937\cdot 10^{-1} + i  9.044873244848521\cdot 10^{-1} ] $$
$$ [\underline{v}_{4},\underline{v}_{5},\underline{v}_{6},\underline{v}_{7}] = [2.231698943133415\cdot 10^{-1} + i 6.687608922343052\cdot 10^{-1}, $$
$$ \quad\quad\quad\quad\quad\quad\quad\,\,6.002906583490685\cdot 10^{-1} + i 2.206993065311951\cdot 10^{-1}, $$
$$ \quad\quad\quad\quad\quad\quad\quad\,\,3.867065923659511\cdot 10^{-1} + i 1.345756657407951\cdot 10^{-1}, $$
$$ \quad\quad\quad\quad\quad\quad\quad\,\,1.526038777608076\cdot 10^{-1}  - i 6.566938139632470\cdot 10^{-1}] $$
$$ [\underline{v}_{8},\underline{v}_{9},\underline{v}_{10},\underline{v}_{11}] = [-1.736311597636939\cdot 10^{-1} -i 7.609350872905872\cdot 10^{-1}, $$
$$ \quad\quad\quad\quad\quad\quad\quad\,\,4.316168188831709\cdot 10^{-1} - i 4.213795287363188\cdot 10^{-1}, $$
$$ \quad\quad\quad\quad\quad\quad\quad\quad-1.741310170829991\cdot 10^{-1} + i 2.308979289387978\cdot 10^{-1}, $$
$$ \quad\quad\quad\quad\quad\quad\quad\quad-1.859590880062348\cdot 10^{-1} + i 2.841353874294199\cdot 10^{-1} ]$$
$$ [\underline{v}_{12},\underline{v}_{13},\underline{v}_{14},\underline{v}_{15}] = [-1.169608247513672\cdot 10^{-1} + i 8.765209704988050\cdot 10^{-1}, $$
$$ \quad\quad\quad\quad\quad\quad\quad\quad5.407734749067091\cdot 10^{-1} - i 3.124118196630669\cdot 10^{-1},$$
$$ \quad\quad\quad\quad\quad\quad\quad\quad1.527995379711511\cdot 10^{-1} + i 6.089233207073497\cdot 10^{-1}, $$
$$ \quad\quad\quad\quad\quad\quad\quad\quad 3.204152865278664\cdot 10^{-1} - i 5.904671666326385 \cdot 10^{-1}] $$
$$ [\underline{v}_{16},\underline{v}_{17},\underline{v}_{18},\underline{v}_{19}] = [-1.330737937092930 \cdot 10^{-1} - i 6.578838258756913\cdot 10^{-1}, $$
$$ \quad\quad\quad\quad\quad\quad\quad\quad\quad-7.385080142922554\cdot 10^{-1} + i 5.831003376892393 \cdot 10^{-1}, $$
$$ \quad\quad\quad\quad\quad\quad\quad\quad 6.699288861915320\cdot 10^{-1} - i 5.533657523766588\cdot 10^{-1}, $$
$$ \quad\quad\quad\quad\quad\quad\quad\quad3.902845736728862\cdot 10^{-1} - i1.485132206413023\cdot 10^{-1} ] $$
$$ [\underline{v}_{20},\underline{v}_{21},\underline{v}_{22},\underline{v}_{23}] = [2.065574523959247 \cdot 10^{-1} - i 6.446211342534475\cdot 10^{-1}, $$
$$ \quad\quad\quad\quad\quad\quad\quad\quad\quad-5.935341284909806\cdot 10^{-1} + i  9.380573668092805\cdot 10^{-1}, $$
$$ \quad\quad\quad\quad\quad\quad\quad\quad\,\, 6.398816151374838\cdot 10^{-1} + i 7.796882650167238\cdot 10^{-1}, $$
$$ \quad\quad\quad\quad\quad\quad\quad\quad\quad-5.451582606670785\cdot 10^{-1}+i  7.278281716630082\cdot 10^{-1} ]$$ 
$$ [\underline{v}_{24},\underline{v}_{25},\underline{v}_{26},\underline{v}_{27}] = [-8.648295254072200\cdot 10^{-1} + i 9.603906925323353\cdot 10^{-1}, $$
$$ \quad\quad\quad\quad\quad\quad\quad\quad1.185927730430883 \cdot 10^{-1} + i 9.581113986558572 \cdot 10^{-1}, $$
$$ \quad\quad\quad\quad\quad\quad\quad\quad\quad-1.272156730720730 \cdot 10^{-1} - i 3.826018264098715 \cdot 10^{-1}, $$
$$ \quad\quad\quad\quad\quad\quad\quad\quad\quad- 6.204082514618516 \cdot 10^{-1} + i 1.303932410632590 \cdot 10^{-1}]$$
$$ [\underline{v}_{28},\underline{v}_{29},\underline{v}_{30},\underline{v}_{31}] = [-4.055191686841270 \cdot 10^{-1} - i 1.061742318012200 \cdot 10^{-1}, $$
$$ \quad\quad\quad\quad\quad\quad\quad\quad 5.497821295556304 \cdot 10^{-1} - i 9.390776649698522 \cdot 10^{-1},$$
$$ \quad\quad\quad\quad\quad\quad\quad\quad\quad-9.573483983309962 \cdot 10^{-1} + i  6.119729292434732 \cdot 10^{-1}, $$
$$ \quad\quad\quad\quad\quad\quad\quad\quad\quad -3.625771545885574 \cdot 10^{-1} + i 5.248770839543854 \cdot 10^{-1}] $$
$$ [\underline{v}_{32},\underline{v}_{33},\underline{v}_{34},\underline{v}_{35}] = [-1.256230679554358 \cdot 10^{-1} - i 1.784554662383838 \cdot 10^{-1}, $$
$$ \quad\quad\quad\quad\quad\quad\quad\quad\quad -2.906342421922776 \cdot 10^{-1} - i 5.154050022761094 \cdot 10^{-1}, $$
$$ \quad\quad\quad\quad\quad\quad\quad\quad\quad -2.436126820291677 \cdot 10^{-1} - i 3.569908300218846 \cdot 10^{-1}, $$
$$ \quad\quad\quad\quad\quad\quad\quad\quad\quad -8.981673995870921 \cdot 10^{-1} + i 1.329775006009424 \cdot 10^{-1}] $$
$$ [\underline{v}_{36},\underline{v}_{37},\underline{v}_{38},\underline{v}_{39}] = [-1.778113408966572 \cdot 10^{-1} + i 1.553732129113020 \cdot 10^{-1}, $$
$$ \quad\quad\quad\quad\quad\quad\quad\quad 2.014979024950735 \cdot 10^{-1} + i 2.767539435584817 \cdot 10^{-1}, $$
$$ \quad\quad\quad\quad\quad\quad\quad\quad\quad -7.897330203466324 \cdot 10^{-1} + i 9.827999531425403 \cdot 10^{-1}, $$
$$ \quad\quad\quad\quad\quad\quad\quad\quad\quad - 3.766244866562784 \cdot 10^{-1} + i 1.708209472557658 \cdot 10^{-1}]. $$}}
\noindent At this point, by replacing in $ F_{(\underline{p}_{0}, \ldots, \underline{p}_{39})}$ the start-parameters with random complex values, 
we get two square polynomial systems $ F_{1}, F_{2} $ of order $ 40 $ and then we construct $ 3 $ segment homotopies $ H_{i} : \C^{40} 
\times  [0,1] \to \C^{40} $ for $ i \in \{0,1,2\} $, so that $ H_{0}$ is between $ F_{(\underline{p}_{0}, \ldots, \underline{p}_{39})} $ and $ F_{1}$, 
$ H_{1} $ connects $ F_{1} $ with $ F_{2} $ and $ H_{2} $ closes the \emph{triangle-loop}. By means of $ H_{0} $, we have a \emph{path} 
from the start-point to a solution of $ F_{1} $ (\emph{endpoint}), which at the second step becomes a start-point for $ H_{1} $, and so on. 
At the end of the triangle-loop, the output is compared with $ \{\underline{v}_{0}, \ldots, \underline{v}_{39}\} $. Since they differ, we restart the 
loop with these two as start-points, and so on. From the $ 6 ^{th}$ iteration on, the number of solutions of $ F_{(\underline{p}_{0}, \ldots, 
\underline{p}_{39})} $ stabilizes on $ 18 $: $ \{\underline{v}_{0}, \ldots, \underline{v}_{39}\} $ plus other $ 17 $ points (we refer to the 
ancillary file {\tt{CatIntConf.pdf}} for the complete list of points and to the ancillary folder {\tt{ternarie\_4444}} for all the files of which our procedure consists of). Therefore $[\underline{p}_{0}, \ldots, \underline{p}_{39}]$ has $ 18 $ decompositions, which allows us to 
conclude the proof.
\end{proof}

As a consequence of Corollary \ref{cor:reduction} and Proposition \ref{prop:decred}, we get the following:
\begin{thm*}\label{thm:4444}
The general polynomial vector $ f $ with $ (n,r;a_{1}, \ldots, a_{r})=(2,4;4,4,4,4) $ admits $ 18 $ simultaneous Waring decompositions 
with  $ 10 $ summands.
\end{thm*}

\begin{rem*}
Theorem \ref{thm:4444} answers one of the open problems introduced in \cite{AngeGaluppiMellaOtt18}. Indeed, thanks to Corollary \ref{cor:reduction}, 
we focused on a square polynomial system of order $ 40 $ instead of the starting one of order $ 60 $, for which our computational technique 
did not provide an answer.  
\end{rem*}

By arguing as in Remark \ref{rem:extcat}, from Theorem \ref{thm:4444} and the observation that the linear space of cubic forms in three variables 
has dimension $10$, we deduce the following:

\begin{cor*}\label{thm:4444ext}
The general polynomial vector $ f = (f_{1},f_{2},f_{3},f_{4}, \ldots, f_{r})$ with $ f_{i} \in \sym^{4}\c^3  $ for $ i \in \{1,2,3,4\} $ and $ f_{i} \in 
\sym^{3}\c^3 $ for $ i \in \{5, \ldots, r \} $ admits $ 18 $ simultaneous Waring decompositions with $ 10 $ summands.
\end{cor*}

As we mentioned in the Introduction, we want to repeat here that computations by Bertini do not provide a theoretical proof that the number 
of decompositions is $18$, 
but they give us a strong evidence of the result. This is the reason for the asterisk  before the statement of the previous results.

\begin{rem} There are indeed two reasons for which the previous computation do not provide a theoretically complete proof that the number of different
decompositions of a polynomial vector of type $(2,4;4,4,4,4) $ is $18$. First, the computation has been performed with a choice of specific vectors $f$ which,
in principle, could be all contained in the Zariski closed locus where the number of decompositions is not general. Second, even if the computation by Bertini
stops after a certain lapse of time by returning $18$ decomposition, there is no theoretical guarantee  that we found all possible decompositions of $f$.\\
For the latter reason, there is nothing to do and we can only refer to the reliability of computations by Bertini, discussed
 in the paper \cite{HauensteinOedOttSomm19}.\\
On the other hand, for the former reason, we can argue that, by \cite{AngeGaluppiMellaOtt18},  the secant variety of the Segre-Veronese 
embedding of $\p^2\times\p^3$ of bi-degree $(1,4)$
is $\p^{59}$, which is a normal variety. Then the Zariski Main Theorem and the Stein Factorization Theorem (\cite{Hartshorne}, p.280)  prove that the number 
of connected components of the variety of the decompositions of
 a generic polynomial vector cannot be smaller that the number of connected components for a specific $f$. 
One can use \cite{COttVan14}, Lemma 2.5 in order to prove that the $18$ decompositions of $f$ that we found in Proposition* 
\ref{prop:decred} are isolated.
\end{rem}

The previous remark implies that the following result (without asterisk) holds.
\begin{thm}  
A general polynomial vector of type $(2,4;4,4,4,4) $ has \emph{at least} $18$ decompositions in terms of polynomial vector of rank $1$.
\end{thm}

\section{Final considerations}

\begin{rem}
After one realizes that catalecticant maps and confinement are useful to determine the number of simultaneous decompositions of $3$ cubics or $4$ quartics, 
one may imagine that  the procedure extends to the study of simultaneous decompositions of $d$ forms of degree $d$ in $3$ variables. 
Unfortunately, this does not happen. For example, the generic rank of $5$ ternary forms of degree $5$ is $15$ (\cite{AngeGaluppiMellaOtt18}).
On the other hand, there are no catalecticant maps providing a confinement for the sets of $15$ points that decompose
a general vector of $5$ quintics. For instance, derivatives of order $4$ determine a map $\C^{15}\to\C^{15}$ which is generically bijective
and covers the whole space spanned by $\p^2\times\p^4$. \smallskip
\end{rem}

We can analyze with catalecticant maps other types of polynomial vectors, if we consider the case in which  the rank is not the generic one,
and the decomposition is not unique. 

\begin{example}\label{ex:sextics}
Let us assume that $ r = 1 $, $ n = 2 $, $ a_{1} = 6 $, i.e. we deal with $ X_{6}^{2} = \nu_{6}(\p^2) \subset \p^{27} $, 
being $ \nu_{6}: \p^2 \to \p^{27} $ the Veronese map of $ \p^2 $ of degree $ 6 $. This case is not perfect, but nevertheless we 
focus on the generic $ f = f_{1} \in \sym^{6}\c^3 $ having the first sub-generic rank, i.e. $ k = 9 $. This is a classical case. 
It is well known (see e.g. Theorem 1.1 of \cite{COttVan17a})
that such $ f $ has two Waring decompositions with $ 9 $ summands, because of the presence of an elliptic 
normal curve of degree $ 3 $ in $ \p^2 $ passing through  the $ 9 $ points of a decomposition of $f$. Indeed, the elliptic cubic maps, 
via the $6$-Veronese embedding, 
to  an elliptic normal curve in $\p^{17}$ containing all the decompositions of $f$, which then are exactly $2$ by Proposition 5.2 of \cite{CCi06}.\\
We show how this example also follows from a catalecticant map.  Consider the \emph{catalecticant map} of the contraction by $f$
$$ C^3_{f}: (\sym^{3}\c^3)^{\vee} \to \sym^{3}\c^3 $$
where $  (\sym^{3}\c^{3})^{\vee} $ can be identified with the space of polynomial derivatives in $ 3 $ variables of order $ 3 $. $ C^3_{f} $ 
is associated to a square matrix of order $ 10 $, with $ 0 $-degree entries, denoted also by $ C^3_{f} $, where in each column there are 
the components, with respect to the standard monomial basis of $ \c[x_{0},x_{1},x_{2}]_{3} $, of $ \partial_{x_{u}}\partial_{x_{v}}
\partial_{x_{w}}f $, with $u,\,v, \, w \in \{0,1,2\} $. More in detail, assuming that $ f = \sum_{i=1}^{9} \lambda_{i}\ell_{i}^6 $, where, 
for simplicity, $ \ell_{i} = x_{0}+\mu_{i}^{1}x_{1}+\mu_{i}^{2}x_{2} $, $ i \in \{1, \ldots, 9\} $, we get that $ C_{f} $ equals
$$ 120 \sum_{i=1}^{9} \lambda_{i} \left( 
 C_{i} |  3\mu_{i}^{1} C_{i} | 3\mu_{i}^{2} C_{i} | 3\mu_{i}^{1}\mu_{i}^{1} C_{i} | 6\mu_{i}^{1}\mu_{i}^{2} C_{i} | 3\mu_{i}^{2}
 \mu_{i}^{2} C_{i} | \mu_{i}^{1}\mu_{i}^{1}\mu_{i}^{1} C_{i} | 3\mu_{i}^{1}\mu_{i}^{1}\mu_{i}^{2} C_{i} | 3\mu_{i}^{1}\mu_{i}^{2}
 \mu_{i}^{2} C_{i} | \mu_{i}^{2}\mu_{i}^{2}\mu_{i}^{2} C_{i} \right) $$
where 
$$ C_{i} = \left( 
1,3\mu_{i}^{1},3\mu_{i}^{2},3\mu_{i}^{1}\mu_{i}^{1},6\mu_{i}^{1}\mu_{i}^{2},3\mu_{i}^{2}\mu_{i}^{2}, \mu_{i}^{1}\mu_{i}^{1}
\mu_{i}^{1},3\mu_{i}^{1}\mu_{i}^{1}\mu_{i}^{2},3\mu_{i}^{1}\mu_{i}^{2}\mu_{i}^{2},\mu_{i}^{2}\mu_{i}^{2}\mu_{i}^{2} \right)^{t} 
\in \nu_{3}([\ell_{i}^{\vee}])  $$
being $ \nu_{3}: \p^2 \to \p^{9} $ the Veronese map of $ \p^2 $ of degree $ 3 $ and $ \ell_{i}^{\vee} = (1, \mu_{i}^{1},\mu_{i}^{2}) $ 
a representative vector for the dual point of $ \ell_{i} $. An explicit computation with Macaulay2 shows that for a generic choice of $ f $ we have that:
\begin{itemize}
\item $ C^3_{f} $ has rank $9$;
\item $\mathcal C_{9} = \p(\im(C^3_{f})) \cap X_{3}^{2}) \subset \p^8 $ is an elliptic curve of degree $9$;
\item $ \p(\im(C^3_{f})) = \langle \nu_{3}([\ell_{1}^{\vee}]), \ldots,\nu_{3}([\ell_{9}^{\vee}]) \rangle. $
\end{itemize}
For a computational verification of these properties, see the ancillary file {\tt{CatIntConf.pdf}}. By semicontinuity, it suffices to prove the above 
properties for a random choice of the $ \ell_{i} $'s. Therefore, for generic choice of $ f $ and for all 
$ i \in \{1, \ldots, 9\} $, the points $ \nu_{3}([\ell_{i}^{\vee}]) \in \mathcal C_{9} = \nu_{3}(\mathcal C_{3}) $, where $ \mathcal C_{3} \subset \p^2 $ is the 
unique elliptic normal curve of degree $ 3 $ passing through the points $ [\ell_{i}^{\vee}] $, according to Theorem 1.1 of \cite{COttVan17a}. 
\end{example}

\begin{example}\label{ex:octics}
Let us assume that $ r = 1 $, $ n = 2 $, $ a_{1} = 8 $, i.e. we deal with $ X_{8}^{2} = \nu_{8}(\p^2) \subset \p^{44} $, 
being $ \nu_{8}: \p^2 \to \p^{44} $ the Veronese map of degree $8$. We 
focus on $ f  \in \sym^{8}\c^3 $ having the first sub-generic rank, i.e. $ k = 14 $. 

By \cite{COttVan17a}, we know that the general $f$ as above is identifiable, i.e. it has a unique
decomposition as a sum of $14$ powers of linear forms.  Nevertheless, by Section 4 of \cite{AngeC} we know that
in the span of $14$ generic powers of linear forms there are forms $g$ which have two decompositions,
the second decomposition formed by linear forms different from the original ones.

By the analysis carried on in \cite{AngeC}, it turns out that for a general choice of $14$ linear forms $L_i$ in $3$
variables (which can be identified as general points in a space $\p^2$), the second decompositions 
of forms $g$ in the span of the set $\{L_1^8,\dots, L_{14}^8\}$ are determined by linear forms $M_1,\dots, M_{14}$
which lie in the unique quartic curve $C$ in $\p^2$ which passes through $L_1,\dots, L_{14}$.

This is is a case of confinement of the possible second decompositions of forms generated by $\{L_1^8,\dots, L_{14}^8\}$,
which is  confirmed by a catalecticant analysis. Namely, the image of the fourth-order derivatives of $f$ span a hyperplane
in the space  $\p^{14}=\p(\sym^{4}\c^3) $, which meets the image of $\p^2$ in the $4$th Veronese map $\nu_4$
in a curve: the curve is exactly $\nu_4(C)$.

Notice that one cannot use the catalecticant map to decide whether or not a form $f$ as above is identifiable. When
the image of the catalecticant map has the correct dimension, then it only depends on the linear forms $L_i$'s.
Thus, the existence of forms like $g$ in the span of $L_1^8,\dots,L_{14}^8$, whose second decompositions cover $C$,
forces the image of the catalecticant map to intersect the Veronese surface   $\nu_4(\p^2)$ in a curve, and not
in a finite set.

Catalecticant maps can guarantee the identifiability of a specific form $f$ only when the span of a minimal
decomposition $A=\{L_1,\dots, L_k\}$ of $f$ only contains forms whose unique minimal decomposition is a subset of $A$.
\end{example}

\begin{rem} A result by  Ballico (\cite{Ball19}, Thm.1.2), in the spirit of the two previous examples,
 provides other cases of confinement of decompositions. Ballico essentially
proves that if a decomposition $A$ of a form $f$ poses independent conditions to forms of degree $d'=\lfloor d/2 \rfloor$, 
then any other decomposition $B$ of $f$ with cardinality smaller or equal than $A$ must lie in the set $X$ cut by forms of degree $d'$
passing through $A$. 
When $X$ is positive dimensional, then it is possible that $f$ has many different decompositions, all of them confined to $X$.
An example of application of this confinement result, for quartics of rank $12$ in $5$ variables, will be the subject of
a forthcoming paper.
\end{rem}

Corollary \ref{thm:collect2} and Corollary \ref{thm:4444ext} are a consequence of a more general observation.

\begin{rem}
Let $ f = (f_{1}, \ldots, f_{r}) $ be a general polynomial vector of rank $ k $ over $ \F $, such that $ f_{j} \in \F[x_0,\ldots,x_n]_{a_{j}} $. 
Let $ f_{r+1}, \ldots, f_{s} $ 
be general elements of $ \F[x_0,\ldots,x_n]_{a_{r+1}} $, with $ a_{r+1} \leq a_{r} $. If $ k \geq \binom{a_{r+1}+n}{a_{r+1}} $, 
then the number of simultaneous Waring decompositions of $ f' = (f_{1}, \ldots, f_{s}) $ with $ k $ summands is greater or equal than the one of $ f $, 
unless a defective case occurs.

Instances of this phenomenon are $ (n,r;a_{1}, \ldots, a_{r}) = (2,3;3,3,3) $  with $k=6$ (as stated in Corollary \ref{thm:collect2}, by adding to a triple of  ternary 
cubics an arbitrary number of conics, the decompositions with six summands of the corresponding multiform are still two), and 
$ (n,r;a_{1}, \ldots, a_{r}) = (2,4;4,4,4,4) $ with $ k = 10 $  
(as stated in Corollary \ref{thm:4444ext}, by adding to four ternary quartics an arbitrary number of cubics, the decompositions with ten 
summands of the corresponding multiform are still $18$).
\end{rem}

\begin{rem}\label{rem:confinement}
Another perfect case in which the catalecticant approach gives a confinement of the decompositions is
 $ (n,r;a_{1}, \ldots, a_{r}) = (3,3;3,3,3) $, with $ k = 10 $. In this situation the decompositions are confined to a general threefold section
 of $\p^2\times\p^3$. The section is again birational to $\p^2\times \p^1$, but the
 map is defined by divisors of type $(4,3)$, with two general triple points, so it  is  different and more complicated than the one of Corollary \ref{cor:i4W}.
 
The number of simultaneous decomposition of $3$ quaternary cubics is $56$. It has been determined, by Bertini,
 in the paper \cite{AngeGaluppiMellaOtt18}. 
 
\end{rem}

\begin{rem}
Let $ f = (f_{1}, \ldots, f_{r}) $ be a general polynomial vector such that $ f_{j} \in \F[x_0,\ldots,x_n]_{a_{1}} $ for any $ j $. The rank of $ f $ over $ \F $ is 
$ k = \left\lceil \frac{r}{n+r}\binom{a_{1}+n}{n}\right\rceil $; see \cite{AngeGaluppiMellaOtt18}. Let $ f_{r+1}, \ldots, f_{r+s} $ be 
general elements of $ \F[x_0,\ldots,x_n]_{a_{2}} $, with $ a_{2} < a_{1} $. 
The rank over $ \F $ of the multiform $ f' = (f_{1}, \ldots, f_{r+s}) $  is given by $ k' =  \left\lceil \frac{1}{n+r+s}\left(r\binom{a_{1}+n}{n}+
s\binom{a_{2}+n}{n}\right)\right\rceil $, \cite{AngeGaluppiMellaOtt18}. 
If $ k \geq k' $, then $ X^{n}_{a_{1}, \ldots, a_{1}, a_{2}, \ldots, a_{2}} $, where $ a_{1} $ and $ a_{2} $ are repeated, respectively, $ r $ and $ s $ times, 
is $k'$-defective. For example, this fact occurs in the following cases:
\begin{itemize}
\item[$\bullet$] $ (n,r;a_{1}) = (2,1;8) $, with $ k = 15, s = 1, a_{2} = 4, k' = 15 $; 
\item[$\bullet$] $ (n,r;a_{1}) = (2,1;12) $, with $ k = 31, s = 1, a_{2} = 5, k' = 28 $;
\item[$\bullet$] $ (n,r;a_{1}) = (2,1;19) $, with $ k = 70, s = 1, a_{2} = 3, k' = 55 $.
\end{itemize}
\end{rem}

 Let us finally notice that the case addressed in Theorem \ref{thm:collect1} is quite peculiar. 
 One can numerically show that there are no other pairs of projective bundles 
$ X^{n}_{d, \ldots, d} $ and $ X^{n}_{d, \ldots, d, d-1} $, with $ n \geq 2 $ and $ d \geq 3 $, providing two perfect cases with 
the same rank $ k $ and $\dim(X_{1, \ldots, 1}^{n} \cap \p^{k-1}) = 1 $ (in the three projective bundles quoted above, the
dots denote that the same degree - $ d $ or $ 1 $ - is repeated $ r-1 $ times).  

\begin{ack} 
The authors are members of the Italian GNSAGA-INDAM and are supported by the Italian PRIN 2015 - 
Geometry of Algebraic Varieties (B16J15002000005).
\end{ack}

\bibliographystyle{amsplain}
\bibliography{biblioLuca}

\end{document}